\documentclass[12pt,a4paper]{amsart}
\usepackage{amssymb}
\usepackage{amsmath}
\usepackage{amsthm}
\usepackage{mathrsfs}
\usepackage{epsfig}
\usepackage{color}

\newtheorem{corollary}{Corollary}
\newtheorem{definition}{Definition}

\newtheorem{lemma}{Lemma}

\newtheorem{proposition}{Proposition}
\newtheorem{remark}{Remark}
\numberwithin{equation}{section}
\newtheorem{theorem}{Theorem}

\def\Z{\mathbb Z}
\theoremstyle{plain}

\newcommand{\R}{\mathbb{R}}

\DeclareMathOperator\supp{supp}
\DeclareMathOperator\card{card}

\def\D{\,\mathrm d}
\def\e{\mathrm e}

\def\supp{\mathop{\rm supp}\nolimits}

\def\spn{\mathop{\rm span}\nolimits}

\def\D{\mathrm{d}}

\begin{document}

\title[On the product in shift-invariant spaces]{On the product of periodic distributions. Product in shift-invariant  spaces}

\author[A. Aksentijevi\' c ]{Aleksandar Aksentijevi\' c }
\address{University of Kragujevac,
Faculty of Science, Department of Mathematics and Informatics, Radoja Domanovi\' ca 12,
34000 Kragujevac,
Serbia}
\email{aksentijevic@kg.ac.rs}

\author[S. Aleksi\' c]{Suzana Aleksi\' c}
\address{University of Kragujevac,
Faculty of Science, Department of Mathematics and Informatics, Radoja Domanovi\' ca 12,
34000 Kragujevac,
Serbia}
\email{suzana.aleksic@pmf.kg.ac.rs}

\author[S. Pilipovi\' c]{Stevan Pilipovi\'{c}}
\address{University of Novi Sad, Faculty of Sciences, Department of Mathematics and Informatics,  Trg Dositeja
Obradovica 4, 21000 Novi Sad, Serbia}
\email{pilipovic@dmi.uns.ac.rs}

\maketitle

\begin{abstract}
We connect through the Fourier transform shift-inva\-riant Sobolev type spaces $V_s\subset H^s$, $s\in\mathbb R,$ and the spaces of periodic distributions and analyze the properties of elements in such spaces with respect to the product. If the series expansions of two periodic distributions have compatible coefficient estimates, then their product is a periodic tempered distribution. We connect product of tempered distributions with the product of shift-invariant elements of $V_s$. The idea for the analysis of products comes from the H\"ormander's description of the Sobolev type wave front in connection with the product of distributions. Coefficient compatibility for the product of $f$ and $g$ in the case of "good" position of their Sobolev type wave fronts is proved in the 2-dimensional case. For larger dimension it is an open problem because of the difficulties on the description of the intersection of cones in dimension $d\geqslant3$.
\end{abstract}

\maketitle

\makeatletter
\renewcommand\@makefnmark%
{\mbox{\textsuperscript{\normalfont\@thefnmark)}}}
\makeatother

\section{Introduction}

Our main interest is the analysis of the product of distributions $f$ and $g$ defined in a neighbourhood of a point $x_0$ belonging to shift-invariant spaces $V_s$, $s\in\mathbb{R}$. We show that, locally, this product is also an element of a shift-invariant space $V_{s_0}$, for some $s_0\in\mathbb{R}$. For this purpose we use the idea of H\" ormander's wave front set (cf.\ \cite{Hor,Hor2}) and the fact that the product exists if the wave fronts are in an appropriate position. Actually, for periodic distributions and shift-invariant distributions we introduce the corresponding notion of compatible coefficient estimates which imply the existence of the product. The results for periodic distributions are transfered to the results for the product in shift-invariant spaces of distributions and vice versa.

Following the range function approach used in Bownik \cite{MB}-\cite{BR} (cf.\ \cite{BVR1}, \cite{BVR2}, \cite{H}, \cite{RS}),  we investigated in \cite{aap} the structure of shift-invariant  subspaces of Sobolev spaces $H^s=H^s(\mathbb R^d)$, $s\in\mathbb R,$ denoted by $V_s$, generated by at most countable family of generators (cf. \cite{AF} for Sobolev spaces). In this paper we consider $V_s$ generated by the finite set of generators,   elements of  $\mathcal A_{s,r}=\{\varphi_1,\ldots,\varphi_r\}\subset H^s$; $V_s$ is the closure of the span of integer translations of functions in $\mathcal A_{s,r}$, $s\in\mathbb{R}$. We will use the notation $V_s(\varphi_1,\ldots,\varphi_r)$ when we want to underline the generators of this space.  In the case $s=0$,  Bownik \cite{MB} gave a comprehensive analysis of the space $V$ ($V=V_0$, $\mathcal{A}_r=\mathcal{A}_{0,r}$). A deep extension of results in \cite{MB} was obtained in \cite{A1}, \cite{A2}, \cite{RS}, \cite{SSS}. Our investigation goes towards the multiplication in spaces $V_s$ for which we need to extend some results concerning the product of periodic distributions.

Let
   $E_s(\mathcal A_{s,r})=\{\varphi_i(\cdot+k):k\in\mathbb Z^d, i=1,\ldots,r\}$ be  a frame of $V_s$ (cf. \cite{KG} for frames). An
$f\in\mathcal{S}^\prime(\mathbb{R}^d)$  belongs to $V_s$ if and only if  its Fourier transform has the form
$\widehat{f}=
 \sum_{i=1}^r f_ig_i$,
 $f_i=\widehat{\varphi}_i\in L_{s}^2(\mathbb{R}^d)$ and $g_i=\sum_{k\in\mathbb{Z}^d}
 {a_k^i}\e^{2\pi \sqrt{-1}\langle\cdot,k\rangle}$ with ${(a^i_k)}_{k\in\mathbb{Z}^d}\in\ell^2$,  $i=1,\ldots,r$.  This is shown in \cite{aap}.
Note that, the products $f_ig_i$, $i=1,\ldots,r$, exist in $L^2_s(\mathbb R^d)=\mathcal F(H^s).$

 Another approach, with the frames consisting of the finite set of generators $\mathcal A_r\subset L^2(\mathbb R^d)$ and expansions with coefficients in  $\ell^2$-sequence space, was developed in \cite{AG}, \cite{AST}, \cite{BLZ}. The spaces with the sequences of coefficients in  $\ell^2_s$ were treated in \cite{SS}, where
 the weights are $(1+|k|^2)^{s/2}$, $k\in\mathbb Z^d$, $s\geqslant 0$, and the finite set of generators are subsets of $ L^2_s=\mathcal F(H^s)$, $s\geqslant 0$.
Actually, in \cite{AG} and \cite{AST}, $\ell^p_s$, $p\geqslant 1$, were considered, but here we restrict ourselves to the   case $p=2$.
  Moreover, connecting two different approaches to shift-invariant spaces $V_s$ and
  $\mathcal V^2_s$, $s>0$, under the assumption that the generators $\varphi_i$, $i=1,\ldots,r$, belong to $ H^s\cap L^2_s$,  we have given the characterization
   of elements in $V_s$  through the expansions with coefficients in $\ell^2_s.$
   The corresponding assertions hold for the intersections of all such spaces and their duals
   in the case when the generators are elements of $\mathcal S(\mathbb R^d)$ (see \cite{aap}).

Our framework in this paper  is the space of periodic distributions, see for example \cite{MPSV} and \cite{LS}, where the authors studied wave fronts through the analysis of Fourier expansions of periodic distributions.   See also  \cite{DV}  and \cite{Ruzh}, where  the authors studied generalized functions on  the $d$-dimensional torus $\mathbb T^d$ and discrete wave fronts.

The paper is organized as follows.  After recalling in Section 2 the basic facts about periodic distributions and shift-invariant spaces, we repeat in Section 3 our results of \cite{aap} concerning different approaches in \cite{AG}, \cite{AST}, \cite{MB}, \cite{SS}, connecting shift-invariant spaces with the subspaces of periodic distributions. In Section 4, after recalling results for the multiplication of periodic distribuctions $f_1$ and $f_2$ belonging to spaces $\mathscr P^{1,s_1}$, $\mathscr P^{2,s_2}$ respectively, we give our main result, Theorem \ref{t7} related to periodic ultradistributions which (according to Definition \ref{d1}) have compatible coefficient estimates. Then, in Theorem \ref{t8} we transfer this result to the product of elements of finitely generated shift-invariant spaces $V_s$. Since our approach was motivated by H\"ormander notion of Sobolev's wave fronts, we devote Section 5 to the product of periodic distributions and of shift-invariant distributions (and vice versa) relating the wave front sets with the compatibility coefficient condition in the expansions of $f_1$ and $f_2$.

 \section{Notation}

Let $x=(x_1,\ldots,x_d)\in\mathbb{K}^d$, where $\mathbb{K}^d\in\{\R^d,\Z^d\}$. We use notations $|x|=\sqrt{x_1^2+\cdots+x_d^2}$ and $\langle x\rangle^s=(1+|x|^2)^{s/2}$, $s\in\mathbb R$. Obviously, $|x|\leqslant\langle x\rangle$. Let $0<\eta\leqslant1$. As in \cite{MPSV}, we use the notation
 $$\mathbb T_{\eta,x}=\prod_{j=1}^{d}\left(x_j-\frac{\eta}2,x_j+\frac\eta2\right) \: \: \mbox{and }\: \: \mathbb T_\eta:= \mathbb T_{\eta,0},\;  \mathbb T=\mathbb T_1.$$
Define the Fourier transform
$\widehat{f}$ of an integrable function $f$ by
$\mathcal{F}f(t)=\widehat{f}(t)=\int_{\mathbb{R}^d}f(x)\e^{-2\pi \sqrt{-1}\langle x,t\rangle}\D x$,
 $t\in\mathbb{R}^d$ ($\mathcal F^{-1}f(t)=\widehat f(-t)$), where
 $\langle x,t\rangle=\sum_{i=1}^dx_it_i$, $x,t\in\mathbb{R}^d$.
Further on,
 $${\ell_s^p=\ell_s^p(\mathbb{Z}^d)=\bigg\{(c_k)_{k\in\mathbb{Z}^d}\,:\,\sum_{k\in\mathbb{Z}^d}|c_k|^p\langle k\rangle^{p\cdot s}<+\infty\bigg\}}, \quad s\in\mathbb R,\, p\geqslant 1.$$ We will consider the case
$p=2$. Then, the scalar product is given by  $\langle (c_k)_{k\in\mathbb{Z}^d},(d_k)_{k\in\mathbb{Z}^d}\rangle_{\ell_s^2}=\sum_{k\in\mathbb{Z}^d}c_k\overline{d}_k\langle k\rangle^{2s}$.

In the sequel we will denote by $C$ constants which are not the same in general; from the context will be clear that in various  inequalities they are different.

\subsection{Periodic distributions}

Our framework is the space of functions and distributions on $\mathbb{R}^{d}$ which are periodic of period $1$ in each variable, i.e.\ $T_nf(x)=f(x-n)=f(x)$, $x\in\R^d, n\in\Z^d$. We refer to the next literature \cite{AMS}, \cite{Beals}, \cite{LS}, \cite{VSV}. Let $x$, $y\in\mathbb{K}^d$. We use notation $e_y(x)=\e^{2\pi \sqrt{-1} \langle y,x\rangle}\;\; (\langle y,x\rangle=\sum_{i=1}^dy_i\overline{x}_i)$.
The space of periodic test functions $\mathscr{P}=\mathscr{P}(\R^d)$ consists of
smooth periodic functions of the form
$\varphi=\sum_{n\in\mathbb Z^d}\varphi_ne_n$ such that
$\sum_{n\in\Z^d}|\varphi_n|^2 \langle n\rangle^{2k}<+\infty$ for every $k\in\Z$
($\varphi_n=\int_{\mathbb T}\varphi(x)e_{-n}(x)\D x$, $n\in\Z^d$);
its topology is given via the sequence of norms  $\|\varphi\|_k=\sup_{x\in \mathbb T, |\alpha|\leqslant k}|\varphi^{(\alpha)}(x)|$,  $k\in\mathbb{N}_0=\mathbb{N}\cup\{0\}$.
The dual space of $\mathscr P$, the space of periodic distributions, is denoted by $\mathscr{P}^\prime$. One has:
$ f=\sum_{n\in\Z^d}{f_n e_n} \in \mathscr{P}^\prime$ if and only if $\sum_{n\in\Z^d}|f_n|^2\langle n\rangle^{-2k_0}<+\infty$, for some $k_0\in\mathbb{N}$. We use notation $\mathscr P'^{k_0}$ when this holds.  If $f=\sum_{n\in\Z^d}{f_n e_n}\in\mathscr{P}'$ and $\varphi=\sum_{n\in\Z^d}{\varphi_n e_n}\in\mathscr{P}$,
 then their dual pairing is given by $\left\langle f,\varphi\right\rangle=\sum_{n\in\Z^d}f_n \varphi_{n}$.

Denote
by $\mathscr P^{p,s}$, $p\geqslant1$, $s\in\R$, the space of elements  $h\in\mathcal D'(\mathbb R^d)$ with the property that  $h=\sum_{n\in\mathbb Z^d}a_ne_n,$ where $ (a_n)_{n\in\Z^d}\in\ell^{p}_s$. These spaces are subspaces of $\mathscr P'$ for $s\leqslant0$. Note, $\bigcap_{s\geqslant0}\mathscr P^{p,s}=\mathscr P$ and $\bigcup_{s\leqslant0}\mathscr P^{p,s}=\mathscr P^\prime.$

Let $x_0\in\mathbb R^d$,  $\psi \in\mathcal D(\mathbb T_{\eta,x_0})$ and   $f\in\mathcal D'(\mathbb R^d)$. Then  $(f\psi)_{per}$ is defined as the periodic extension,
 by $(f\psi)_{per}(t)=(f\psi)(x),$ where $t+k=x\in \mathbb T_{\eta,x_0}$, $k\in\mathbb Z^d$
(this $k$ is unique). So,
$$(f\psi)_{per}(t)=\sum_{k\in\mathbb Z^d}a_ke_k(t),\quad t\in\mathbb R^d,$$
where $a_k=\int_{\mathbb T_{\eta,x_0}}(f\psi)(t)e_{-k}(t)\D t$,  $k\in\mathbb Z^d$.
We denoted by $\mathscr{P}_{loc}^{p,s}$ the local space which contains distributions $f\in\mathcal{D}^\prime(\R^d)$ such that $(f\psi)_{per}\in\mathscr P^{p,s}$, for all $x_0\in\R^d$ and $\psi\in\mathcal{D}(\mathbb{T}_{1,x_0})$. In particular, we consider the cases $p=1,2$.

\subsection{Shift-invariant spaces}
  Recall (\cite{aap}),
 the Hilbert space $H(\mathbb{T},\ell^2_{s})$  consists of all vector valued measurable square integrable  functions $F:\mathbb{T}\to\ell_s^2$ with the norm
 $\|F\|_{H(\mathbb{T},\ell_s^2)}=\big(\int_{\mathbb{T}}\|F(t)\|_{\ell_s^2}^2\D t\big)^{1/2}<+\infty.$
In the case $s=0,$ it is denoted by $L^2(\mathbb{T},\ell^2)$.
If $\mathcal{A}_r\subset L^2(\mathbb{R}^d)$, then $\mathcal{A}_{s,r}=\{\varphi\in \mathcal{S}^\prime(\mathbb{R}^d):\widehat{\varphi}=\widehat{\psi}\langle\cdot\rangle^{-s} \;\mbox{for some } \;\psi\in\mathcal{A}_r\}$.

Note that any space $W\subset H^s$ is called shift-invariant if\ $\varphi\in W$ implies $T_k \varphi\in W$, for any $k\in \mathbb{Z}^d$.
We define $V_s=\overline{\spn}\{(1-\tfrac{\Delta}{4\pi^2})^{- s/2}T_k\psi: \psi\in\mathcal{A}_r, k\in\mathbb{Z}^d\}$, where $\Delta$ is the Laplacian. It is a shift-invariant space.

Following the definition of the mapping $\mathcal{T}:L^2\to L^2(\mathbb{T},\ell^2)$ (\cite{MB}), we define in \cite{aap}, ${\mathcal{T}_s}:H^s \to H(\mathbb{T},\ell_s^2)$
  ($\mathcal T=\mathcal T_s$, for $s=0$) by
 $${\mathcal{T}_s} \varphi(t)=\bigg(\frac{\widehat{\psi}(t+k)}{\langle k\rangle^s}\bigg)_{k\in\mathbb{Z}^d},\quad t\in\mathbb{T},\; \varphi\in H^s,$$
 where $\big(1-\frac{\Delta}{4\pi^2}\big)^{s/2}\varphi=\psi (\in L^2(\mathbb R^d))$.

 \begin{lemma} [\cite{aap}]\label{pomoc} Let $s\in\mathbb R.$  
\begin{itemize}
\item[a)] ${\mathcal{T}}_s: H^s\rightarrow H(\mathbb T,\ell^2_s)$ is an isometric isomorphism.
\item[b)] The following diagram of isometries  commutes
  \hspace*{3cm}\begin{align*}\hspace*{5cm}&L^2\quad \xrightarrow{\mathcal{T}}\quad L^2(\mathbb{T},\ell^2)\\
  &\downarrow \alpha_s\hspace*{1.5cm}\downarrow \beta_s\\
  &H^s \quad\xrightarrow{\mathcal{T}_s} \quad H(\mathbb{T},\ell_s^2),
  \end{align*} where\, $\alpha_s(g)=\mathcal{F}^{-1}\big(\frac{\widehat{g}(\cdot)}{\langle\cdot\rangle^s}\big)$\, and\, $\beta_s\big((f_k(\cdot))_{k\in\mathbb{Z}^d}\big)=
  \big(\frac{f_k(\cdot)}{\langle k\rangle^s}\big)_{k\in\mathbb{Z}^d}$; in particular,
  $\beta_s\big((\widehat{g}(\cdot+k))_{k\in\mathbb{Z}^d}\big)=\big(\frac{\widehat{g}(\cdot+k)}{\langle k\rangle^s}\big)_{k\in\mathbb{Z}^d}$.
\item[c)] Let $\varphi\in\mathcal S(\mathbb R^d).$ Then
 ${\mathcal{T}}_sT_j \varphi(\cdot)=e_{-j}(\cdot)                                                                                                                                                                                                                                                                                                                                                                                                                                                                                                                                                                                                                                                                                                                                                                                                                                                                                                                                                                                   {\mathcal{T}}_s \varphi(\cdot)$, $j\in\mathbb Z^d.$
 \end{itemize}
  \end{lemma}

\section{Structural theorems}

We introduce the following assumptions on generators $\psi^i$, $i=1,\ldots,r,$ of $V_s(\psi^1,\ldots,\psi^r)$,  in order to have that their linear
 combinations determine subspaces of $H^s$ and of $L^2_s$:
\begin{equation}\label{join}
 \psi^i\in H^s\cap L^2_s\cap  \mathcal{L}^\infty,\quad i=1,\ldots,r.
\end{equation}
 Recall \cite{AST}, the Wiener amalgam type space, denoted by $\mathcal{L}^\infty$ is defined by
$$\mathcal{L}^{\infty}=\bigg\{\psi:\|\psi\|_{\mathcal{L}^\infty}=\sup_{t\in\mathbb T}\sum_{j\in\mathbb{Z}^d}|\psi(t+j)|<+\infty\bigg\}$$
and following this paper in the case $p=2,$ in \cite{SS} is defined:
\begin{equation}
\label{001}\mathcal{V}_s^2=\bigg\{f: f=\sum_{i=1}^r\sum_{k\in\mathbb{Z}^d}c_k^iT_{-k}\psi^i,\; (c_k^i)_{k\in\mathbb{Z}^d}\in\ell_s^2,\;  i=1,\ldots,r\bigg\}.
\end{equation}

\begin{theorem} [\cite{aap}] \label{jed}
Let $s\geqslant 0$, and  $(\ref{join})$ hold.
\begin{itemize}\item[a)] Assume that
$\mathcal V^2_s$ and $\mathcal F(\mathcal V_s^2)$ are closed  in $L^2_s$. Then,
 $\mathcal{V}_s^2\subset H^s$  and  $\mathcal{V}_s^2=V_s(\psi^1,\ldots,\psi^r).$
In particular,
any element  $f\in {V}_s(\psi^1,\ldots,\psi^r)$ has the frame expansion as in $(\ref{001})$.
\item[b)] Assume that $s>\frac12$ and that
$\mathcal V^2_s$ is closed  in $L^2_s$. Then, $ \mathcal{F}(\mathcal V_s^2)$ is closed  in $L^2_s$ and both assertions
in $a)$ hold true.
\item[c)] Assume that the conditions of assertion $a)$ or conditions of assertion $b)$ hold. Then in $($both cases$)$,
\begin{itemize}
\item[(i)]  $(\mathcal V^2_s)'= \mathcal V^2_{-s}$, where $\mathcal V^2_{-s}$ is the space of formal series of the form
$$\hspace*{1cm}F(\cdot)=\sum_{i=1}^r\sum_{k\in\mathbb Z^d} b^i_k\psi^i(\cdot+k)\quad\text{such that}\quad\sum_{i=1}^r\sum_{k\in\mathbb Z^d} |b^i_k|^{2}\langle k\rangle^{-2s}<+\infty,
$$
with the dual pairing
$\langle F,f\rangle=\sum_{i=1}^r\sum_{k\in\mathbb Z^d} b^i_k c^i_k$,  $(f$ is of the form given in $(\ref{001})).$
\item[(ii)] $\mathcal V^2_{-s}=V_{-s}.$
\end{itemize}
\end{itemize}
\end{theorem}

\begin{theorem}[\cite{aap}]
Assume  $\psi^i\in\mathcal S(\mathbb R^d)$, $i=1,\ldots,r$. Then,
$\bigcap_{s\geqslant 0}\mathcal{V}_s^2=\bigcap_{s\geqslant 0}V_s$
and the expansion for their elements has the form as in  $(\ref{001})$ with
$$\sup_{k\in\mathbb Z^d}|c^i_k||k|^s<+\infty, \quad i=1,\ldots,r, \mbox{ for every } s>0.$$
Moreover, $\mathcal{F}\big(\bigcap_{s\geqslant 0}\mathcal{V}_s^2\big)
=\big\{\sum_{i=1}^r\widehat{\psi}^i(\cdot)\Phi_i(\cdot):\Phi_i\in\mathscr P\big\}$,
where $\Phi_i(\cdot)=\sum_{k\in\mathbb{Z}^d}c_k^ie_k(\cdot)$, $(c_k^i)_{k\in\mathbb{Z}^d}\in\ell_s^2$ for every $s\geqslant0$, $i=1,\ldots,r$, and
 $V_s^\prime=\mathcal V^2_{-s}$,  $\bigcup_{s\geqslant0}V_s^\prime=\bigcup_{s\geqslant0}\mathcal V^2_{-s}$. Also,
$\mathcal{F}\big(\bigcup_{s\leqslant 0}\mathcal{V}_s^2\big)=\big\{\sum_{i=1}^r\widehat{\psi}^i(\cdot)F_i(\cdot): F_i\in\mathscr P^\prime\big\}$,
 where $F_i(\cdot)=\sum_{k\in\mathbb{Z}^d}c_k^ie_k(\cdot)$, $(c_k^i)_{k\in\mathbb{Z}^n}\in\ell_s^2$ for some $s\leqslant 0$, $i=1,\ldots,r.$
\end{theorem}

\section{Multiplication}\label{multiplication}

Let  $f_1=\sum_{n\in\mathbb Z^d}f_{1,n}e_n\in\mathscr{P}^{1,s}$ and $f_2=\sum_{n\in\mathbb Z^d}f_{2,n}e_n\in\mathscr{P}^{2,s}$.
Their  product is defined  as $$f=f_{1}f_{2}:=\sum_{n\in\mathbb{Z}^d}f_{n}e_{n}, \mbox{ where } f_n=\sum_{j\in\mathbb{Z}^{d}}f_{1,n-j}f_{2,j}, \;n\in\Z^d.$$
Then (\cite{MPSV}), $f\in\mathscr{P}^{2,s}$ and the  mapping
$$\mathscr{P}^{1,s}\times\mathscr{P}^{2,s}\ni (f_{1},f_{2})\mapsto f_{1}f_{2} \in \mathscr{P}^{2,s}$$
is continuous. If $s$, $s_1$, $s_2\in\R$ satisfy
$s_1+s_2\geqslant0$ and $s\leqslant\min\{s_1,s_2\}$, then the mapping
\begin{equation}\label{4.1}\mathscr{P}^{1,s_1}\times\mathscr{P}^{2,s_2}\ni (f_{1},f_{2})\mapsto f_{1}f_{2} \in \mathscr{P}^{2,s}
\end{equation} is continuous.

This implies the following assertion.
\begin{proposition}
Let $f_1(\cdot)=\sum_{k\in\mathbb Z^d}f_{1,k}\varphi(\cdot+k)$, so that $(f_{1,k})_{k\in\Z^d}\in\ell^{1}_{s_1}$,
$f_2(\cdot)=\sum_{k\in\mathbb Z^d}f_{2,k}\phi(\cdot+k)$, so that $(f_{2,k})_{k\in\Z^d}\in\ell^{2}_{s_2}$, where $\varphi\in H^{s_1}$, $\phi\in H^{s_2}\cap\mathcal{F}^{-1}\big(L^\infty(\R^d)\big)$ and $s_1+s_2\geqslant0$. Then $f=f_1*f_2\in V_s$, where $s\leqslant\min\{s_1,s_2\}$, is  generated by $\varphi*\phi\in H^s$, i.e. \[f(\cdot)=\sum_{n\in\mathbb Z^d}f_n(\varphi*\phi)(\cdot+n),\] where $(f_n)_{n\in\Z^d}\in\ell^2_s$.
\end{proposition}
\begin{proof} By the assumptions, $\widehat{\varphi}\widehat{\phi}\in L_s^2(\R^d)$. Since
$$\widehat{f}_1\widehat{f}_2=\widehat{\varphi}\widehat{\phi}\sum_{n\in\Z^d}f_ne_n,$$
where $f_n=\sum_{j\in\Z^d}f_{1,n-j}f_{2,j}$, $n\in\mathbb{Z}^d$ belongs to $\ell^2_s$, by \eqref{4.1} one has that
\begin{align*}f(t)&=(f_1*f_2)(t)\\&=(\varphi*\phi)(t)*\sum_{n\in\Z^d}f_n\delta(t+n)\\&=\sum_{n\in\Z^d}f_n(\varphi*\phi)(t+n),\quad t\in\R^d,\end{align*}
whence the assertion follows.
\end{proof}

The previous considerations allow us to introduce multiplication in the local versions of these spaces.
Let $f_1\in\mathscr P_{loc}^{1,s}$ and $f_{2}\in\mathscr{P}_{loc}^{2,s}$. To define their product $f=f_{1}f_{2}$, we proceed locally. Let $x_{0}\in\mathbb{R}^{d}$ and $0<\eta<1$. Let
$\phi\in \mathscr{D}(\mathbb T_{1,x_{0}})$ be such that $\phi(x)=1$ for
$x\in \mathbb T_{\varepsilon,x_{0}}$, $\varepsilon<\eta$. We define
$f_{\eta,x_{0}}$ as the restriction to $\mathbb T_{\eta,x_{0}}$ of the product $(\phi f_{1})_{per} (\phi f_{2})_{per}$.
So, $f_{\eta,x_{0}}\in\mathscr{D}^\prime(\mathbb T_{\eta,x_{0}})$.
By the use of the partition of unity the authors of \cite{MPSV} have the next assertion.
\begin{corollary}[\cite{MPSV}] The product $f=f_1f_2$ of $f_1\in\mathscr P_{loc}^{1,s_1}$ and $f_2\in\mathscr P_{loc}^{2,s_2}$ is an element of $\mathscr P_{loc}^{2,s}$, where $s_1+s_2\geqslant0$ and $s\leqslant\min\{s_1,s_2\}$. Moreover, the mapping
$$\mathscr P_{loc}^{1,s_1}\times\mathscr P_{loc}^{2,s_2}\ni (f_1,f_2)\mapsto f_1f_2=f \in\mathscr P_{loc}^{2,s}$$
is continuous.
\end{corollary}

Now we consider the product of two periodic distributions.

\begin{theorem}\label{t7} Let $f_1$, $f_2\in\mathscr{P}^\prime$, i.e.
$$f_1=\sum_{i=1}^{l_1}\sum_{k\in\mathbb{Z}^d}a^i_{1,k}e_k,\quad
f_2=\sum_{j=1}^{l_2}\sum_{k\in\mathbb{Z}^d}a_{2,k}^je_k,$$
such that there exist sets $\Lambda_i^1$, $i=1,\ldots,l_1$, and  $\Lambda_j^2$, $j=1,\ldots,l_2$, subsets of $\mathbb{Z}^d$ so that
\begin{equation}\label{5.4}\sum_{k\in\Lambda_i^1}|a^i_{1,k}|^2\langle k\rangle^{-2\alpha_1}<+\infty, \quad
\sum_{k\in\mathbb{Z}^d\setminus\Lambda_i^1}|a^i_{1,k}|^2\langle k\rangle^{2\beta_1}<+\infty,
\end{equation}
\begin{equation}\label{5.5} \sum_{m\in\Lambda_j^2}|a^j_{2,m}|^2\langle m\rangle^{-2\alpha_2}<+\infty,
\quad\sum_{m\in\mathbb{Z}^d\setminus\Lambda_j^2}|a^j_{2,m}|^2\langle m\rangle^{2\beta_2}<+\infty,
\end{equation} for $i=1,\ldots,l_1$, $j=1,\ldots,l_2$, some $\beta_1\geqslant\alpha_2\geqslant0$, $\beta_2\geqslant\alpha_1\geqslant0$, and $\Lambda_i^1\cap(-\Lambda_j^2)=\emptyset$, $i=1,\ldots,l_1$, $j=1,\ldots,l_2$.
Moreover, we assume that for every $i=1,\ldots,l_1$, for every $j=1,\ldots,l_2$ and every $n\in\mathbb{Z}^d$, there exist $C>0$ and $\gamma\geqslant1$ such that
\begin{equation}\label{5.6}c_{i,j}^1(n)=\card\{k\in\mathbb{Z}^d\,:\,n-k\in\Lambda_j^2\, \wedge \, k\in\Lambda_i^1\}\leqslant C |n|^{\gamma}.
\end{equation}
Then there exists $\tau\in\mathbb R$ such that  $f_1f_2\in\mathscr{P}'^\tau$.
\end{theorem}

We give the proof of Theorem \ref{t7} immediately after the following definition, which seems reasonable.
\begin{definition}\label{d1} It is said that $f_1$, $f_2\in\mathscr P^\prime$ have compatible coefficient estimates if \eqref{5.4}-\eqref{5.6} hold.
We say that $f_1$, $f_2\in\mathcal{D}^\prime(\R^d)$  have compatible coefficient estimates in a neighborhood of $x_0$ if for some $\varphi\in\mathcal{D}(\mathbb{T}_{\eta,x_0})$, $(f_1\varphi)_{per}$ and $(f_2\varphi)_{per}$ have Fourier expansions so that \eqref{5.4}-\eqref{5.6} hold.
 The sequences $(a_{1,k}^i)_{k\in\Z^d}$, $i=1,\ldots,l_1$, and $(a_{2,k}^j)_{k\in\Z^d}$, $j=1,\ldots,l_2$, are compatible if \eqref{5.4}-\eqref{5.6} hold.
\end{definition}

\begin{proof}[Proof of Theorem \ref{t7}] First, we note that if $c_{j,i}^2(n)=\card\{k\in\mathbb{Z}^d\,:\,n-k\in\Lambda_i^1\, \wedge\, k\in\Lambda_j^2\}$, then $c_{i,j}^1(n)=c_{j,i}^2(n)$.

The proof will be given for $l_1=l_2=1$. The transfer to the general case is just repetition of arguments which are to follow.
So,  we will cancel indexes $i$ and $j$. Thus, we have
\begin{align*}f_1f_2&=\Bigg(\sum_{k\in\Lambda^1}+\sum_{k\in\mathbb{Z}^d\setminus\Lambda^1}\Bigg)a_{1,k}e_k\cdot
\Bigg(\sum_{m\in\Lambda^2}+\sum_{m\in\mathbb{Z}^d\setminus\Lambda^2}\Bigg)a_{2,m}e_m\\&=f_1^1f_2^1+f_1^1f_2^2+f_1^2f_2^1+f_1^2f_2^2,\end{align*}
and assume that
$$2\tau\geqslant\max\big\{4\gamma(\alpha_1+\alpha_2)+2\gamma+d+1,2\alpha_1+d+1,2\alpha_2+d+1\big\}.$$

We will estimate separately all the summaries. We have
$$f_1^1f_2^1=\sum_{n\in\mathbb{Z}^d}a_n^{11}e_n,\quad\text{where}\quad a_n^{11}=\sum_{n-k\in\Lambda^1\atop k\in\Lambda^2}a_{1,n-k}a_{2,k}, \; n\in\mathbb{Z}^d.$$
There holds,
\begin{align*}&\sum_{n\in\mathbb{Z}^d}|a_n^{11}|^2\langle n\rangle^{-2\tau}
\leqslant\\&\leqslant\sum_{n\in\mathbb{Z}^d}\Bigg(\sum_{n-k\in\Lambda^1\atop k\in\Lambda^2}|a_{1,n-k}|\langle n-k\rangle^{-\alpha_1}|a_{2,k}|\langle k\rangle^{-\alpha_2}\cdot\langle n-k\rangle^{\alpha_1}\langle k\rangle^{\alpha_2}\Bigg)^2\langle n\rangle^{-2\tau}.\end{align*}
By \eqref{5.6}, for $k\in\Lambda^2$ and $(n-k)\in\Lambda^1$,
$$\langle n-k\rangle\leqslant\big\langle (n_1+|n|^\gamma,\ldots,n_d+|n|^\gamma)\big\rangle\leqslant C\langle n\rangle^{2\gamma},\quad
\langle k\rangle\leqslant C\langle n\rangle^{2\gamma}.$$
We continue,
\begin{align*}&\sum_{n\in\mathbb{Z}^d}|a_n^{11}|^2\langle n\rangle^{-2\tau}
\leqslant\\
 &\leqslant C\sum_{n\in\mathbb{Z}^d}\Bigg(\sum_{n-k\in\Lambda^1\atop k\in\Lambda^2}|a_{1,n-k}|\langle n-k\rangle^{-\alpha_1}|a_{2,k}|\langle k\rangle^{-\alpha_2}\Bigg)^2\frac{\langle n\rangle^{4\gamma(\alpha_1+\alpha_2)+2\gamma}}{\langle n\rangle^{2\tau}}\\
&\leqslant C\sum_{n\in\mathbb{Z}^d}\Bigg(\sum_{n-k\in\Lambda^1\atop k\in\Lambda^2}|a_{1,n-k}|^2\langle n-k\rangle^{-2\alpha_1}\Bigg)
\Bigg(\sum_{n-k\in\Lambda^1\atop k\in\Lambda^2}|a_{2,k}|^2\langle k\rangle^{-2\alpha_2}\Bigg)\frac1{\langle n\rangle^{d+1}}\\&\leqslant C\sum_{n\in\mathbb{Z}^d}\frac1{\langle n\rangle^{d+1}}<+\infty.
\end{align*}

Let us estimate
$$f_1^1f_2^2=\sum_{n\in\mathbb{Z}^d}a_n^{12}e_n,\quad \text{where}\quad
a_n^{12}=\sum_{n-k\in\Lambda^1\atop k\in\mathbb{Z}^d\setminus\Lambda^2}a_{1,n-k}a_{2,k},\quad n\in\mathbb{Z}^d.$$
There holds,
\begin{align*}&\sum_{n\in\mathbb{Z}^d}|a_n^{12}|^2\langle n\rangle^{-2\tau}
\leqslant\\
&\leqslant\sum_{n\in\mathbb{Z}^d}\Bigg(\sum_{n-k\in\Lambda^1\atop k\in\mathbb{Z}^d\setminus\Lambda^2}|a_{1,n-k}|\langle n-k\rangle^{-\alpha_1}|a_{2,k}|\langle k\rangle^{\beta_2}\cdot\frac{\langle n-k\rangle^{\alpha_1}}{\langle k\rangle^{\beta_2}}\Bigg)^2\langle n\rangle^{-2\tau}.
\end{align*}
Since, $\langle n-k\rangle^{\alpha_1}\leqslant C\langle n\rangle^{\alpha_1}\langle k\rangle^{\alpha_1}$ and $\beta_2\geqslant\alpha_1\geqslant0$, we have
$$\sum_{n\in\mathbb{Z}^d}|a_n^{12}|^2\langle n\rangle^{-2\tau}\leqslant C\sum_{n\in\mathbb{Z}^d}\frac1{\langle n\rangle^{2\tau-2\alpha_1}}<+\infty.$$

Further, we will use the inequality
\begin{equation}\label{ineq.}\langle y\rangle^r\leqslant C\langle x\rangle^r\langle y-x\rangle^{|r|},\quad x,y\in\R^d, \,r\in\R,
\end{equation} which we now show that hold. Indeed, since \[(1+t)^2=1+2t+t^2\leqslant2(1+t^2),\quad t\geqslant0,\] if we choose $t=|y-x|$ then we get
$$\langle y\rangle\leqslant\langle x\rangle+|y-x|\leqslant\langle x\rangle(1+|y-x|)\leqslant2^{1/2}\langle x\rangle\langle y-x\rangle.$$
Thus, for $r\geqslant0$ inequality \eqref{ineq.} holds. For $r<0$ we have
$$\frac{\langle y\rangle^r}{\langle x\rangle^r}=\frac{\langle x\rangle^{|r|}}{\langle y\rangle^{|r|}}\leqslant\frac{C\langle y\rangle^{|r|}\langle y-x\rangle^{|r|}}{\langle y\rangle^{|r|}}=C\langle y-x\rangle^{|r|}.$$
Hence, the inequality \eqref{ineq.} holds for every $r\in\R$.

Now, by \eqref{ineq.} for $k$, $n\in\Z^d$ and $\alpha_2\geqslant0$, we have that $\langle k\rangle^{\alpha_2}\leqslant C\langle k-n\rangle^{\alpha_2}\langle n\rangle^{\alpha_2}$ holds.
Using the last inequality and $\beta_1\geqslant\alpha_2\geqslant0$, the estimate for $f_1^2f_2^1$ simply follows:
\begin{align*}&\sum_{n\in\mathbb{Z}^d}|a_n^{21}|^2\langle n\rangle^{-2\tau}
\leqslant\\&\leqslant\sum_{n\in\mathbb{Z}^d}\Bigg(\sum_{n-k\in\Z^d\setminus\Lambda^1\atop k\in\Lambda^2}|a_{1,n-k}|\langle n-k\rangle^{\beta_1}|a_{2,k}|\langle k\rangle^{-\alpha_2}\cdot\frac{\langle k\rangle^{\alpha_2}}{\langle n-k\rangle^{\beta_1}}\Bigg)^2\langle n\rangle^{-2\tau}\\
&\leqslant C\sum_{n\in\mathbb{Z}^d}\Bigg(\sum_{n-k\in\Z^d\setminus\Lambda^1\atop k\in\Lambda^2}|a_{1,n-k}|^2\langle n-k\rangle^{2\beta_1}\Bigg)\Bigg(\sum_{n-k\in\Z^d\setminus\Lambda^1\atop k\in\Lambda^2}|a_{2,k}|^2\langle k\rangle^{-2\alpha_2}\Bigg)\frac{\langle n\rangle^{2\alpha_2}}{\langle n\rangle^{2\tau}}
\\&\leqslant C\sum_{n\in\Z^d}\frac1{\langle n\rangle^{d+1}}<+\infty.
\end{align*}
The estimate for $f_1^2f_2^2$ is proved in a similar way. Hence, $f_1f_2\in\mathscr P^{\prime\tau}$.
\end{proof}

\begin{theorem}\label{t8} Let $g_1\in V_{s_1}(\varphi_1^1,\ldots,\varphi_1^{l_1}),$ $g_2\in V_{s_2}(\varphi_2^1,\ldots,\varphi_2^{l_2})$, $s_1, s_2\geqslant0$, so that
$$g_1(\cdot)=\sum_{i=1}^{l_1}\sum_{k\in\mathbb{Z}^d}a^i_{1,k}\varphi^i_1(\cdot+k),\quad g_2(\cdot)=\sum_{j=1}^{l_2}\sum_{k\in\mathbb{Z}^d}a^j_{2,k}\varphi^j_2(\cdot+k),$$
and that there exist sets $\Lambda_i^1$, $i=1,\ldots,l_1$, and $\Lambda_j^2$, $j=1,\ldots,l_2,$ subsets of $\mathbb{Z}^d$ such that $\Lambda_i^1\cap(-\Lambda_j^2)=\emptyset$, $i=1,\ldots,l_1$, $j=1,\ldots,l_2$. Moreover, assume that \eqref{5.4} and \eqref{5.5} hold and that $c_{i,j}^1(n)$ $(c_{j,i}^2(n))$, $i=1,\ldots,l_1$, $j=1,\ldots,l_2$, satisfy \eqref{5.6}. Then, there exists $s\in\mathbb R$ such that for  $\varphi^i_1$, $\varphi^j_2\in V_s\cap\mathcal{V}_s^2$, $i=1,\ldots,l_1$, $j=1,\ldots,l_2$, 
we have
$$g_1*g_2\in V_s\big(\varphi^i_1*\varphi^j_2,\, i=1,\ldots,l_1,\, j=1,\ldots,l_2\big).$$
More precisely,
$$(g_1*g_2)(\cdot)=\sum_{i=1}^{l_1}\sum_{j=1}^{l_2}\sum_{n\in\Z^d}\sum_{n-k\in\Z^d}a_{1,n-k}^ia_{2,k}^j(\varphi_1^i*\varphi_2^j)(\cdot+n).$$
\end{theorem}
\begin{proof} Again, we discuss only the case $l_1=l_2=1$ and cancel indices $i$ and $j.$ We have
$$\widehat{g}_1(x)=\widehat{\varphi}_1(x)f_1(x),\quad \widehat{g}_2(x)=\widehat{\varphi}_2(x)f_2(x),\quad x\in\mathbb{R}^d,$$
where $f_1$ and $f_2$ are the same as in Theorem \ref{t7}. Thus,
$$\widehat{g}_1(x)\widehat{g}_2(x)=\widehat{\varphi}_1(x)\widehat{\varphi}_2(x)\sum_{n\in\mathbb{Z}^d}a_ne_n(x),\quad x\in\R^d,$$
and coefficients $a_n$, $n\in\mathbb Z^d$ satisfy
 $$\sum_{n\in\mathbb{Z}^d}|a_n|^2\langle n\rangle^{-2\tau}<+\infty,$$ by Theorem \ref{t7}. This implies
$$(g_1*g_2)(t)=(\varphi_1*\varphi_2)(t)*\sum_{n\in\mathbb{Z}^d}a_n\delta(t+n)=\sum_{n\in\mathbb{Z}^d}a_n(\varphi_1*\varphi_2)(t+n),\quad t\in\R^d.$$
Let $s=-\tau$. Hence, $g_1*g_2\in V_s(\varphi_1*\varphi_2)$.

The general case is again the repetition of the given proof but with much more complex notation which we skip.
\end{proof}

\section{Wave front characterizations}

 We analyze the product of distributions considering them in the space of periodic distributions and then we transfer the obtained results to the shift-invariant spaces $V_s.$
 We recall H\"ormander's definition \cite{Hor2}.
\begin{definition}[\cite{Hor2}]
Let $f\in\mathscr{D}'(\mathbb{R}^{d})$, $(x_{0},\xi_0)\in\mathbb{R}^{d}\times(\R^d\setminus\{\bf{0}\})$, and $s\in\mathbb{R}$. We say that $f$ is Sobolev microlocally
regular at $(x_0,\xi_0)$ of order $s$, that is $(x_0,\xi_0)\notin WF_s(f)$, if there exist an open cone $\Gamma$ around $\xi_0$ and $\psi\in \mathscr D(\mathbb{R}^{d})$ with $\psi\equiv 1$ in a neighborhood of $x_{0}$ such that
$$\int_\Gamma{|\widehat{\psi f}(\xi)|^2\langle \xi\rangle^{2s} \D\xi<+\infty}.$$
\end{definition}

The next theorem is  the characterization through the localization and the representation through the Fourier coefficients.
 \begin{theorem}[\cite{MPSV}]\label{sobwf}
Let $f \in \mathscr{D}^\prime(\R^d)$. The following two conditions are equivalent.
\begin{itemize}
\item[a)] There exist an open cone $\Gamma$ around $\xi_0$, $\phi\in \mathscr D(\mathbb T_{\eta,x_{0}})$ with $\eta\in(0,1)$ and $\phi\equiv 1$ in a neighborhood of $x_{0}$, such that
$$\sum_{n\in\Gamma\cap\Z^d}{|a_n|^2\langle n\rangle^{2s}}<+\infty,\quad \mbox{where }\;(f\phi)_{per}=\sum_{n\in\mathbb{Z}^{d}} a_n e_n.$$
\item[b)] $(x_0,\xi_0)\notin WF_s(f).$
\end{itemize}
\end{theorem}

The next assertion is interesting in itself.
\begin{theorem}\label{t6*} Let $\Gamma$ be an open convex cone in $\R^d\setminus\{\mathbf{0}\}$ and $f=\sum_{n\in\Z^d}a_ne_n\in\mathscr P^\prime$, so that $\sum_{n\in\Gamma\cap\Z^d}|a_n|^2\langle n\rangle^{2s}<+\infty$. Then $(x_0,\xi_0)\notin WF_s(f)$ for any $x_0\in\R^d$ and $\xi_0\in\Gamma$.
\end{theorem}
\begin{proof} Let $\varphi\in\mathcal{D}(\mathbb{T}_{\eta,x_0})$, $\varphi\equiv1$ in $\mathbb{T}_{\varepsilon,x_0}$, $0<\varepsilon<\eta$. We know that $\widehat{\varphi}\in\mathcal{S}(\R^d)$. Let $\Gamma_{\xi_0}\subset\Gamma$ and $\Gamma_1\subset\subset\Gamma_{\xi_0}$ (that is, $\Gamma_1\cap\mathbb{S}^{d-1}$ is a compact subset of $\Gamma_{\xi_0}\cap\mathbb{S}^{d-1}$, where $\mathbb{S}^{d-1}$ is the unit sphere). Then there exists $C>0$ such that
\begin{equation}\label{1*}\xi\in\Gamma_1\quad\wedge\quad n\in\Z^d\cap\big((\R^d\setminus\{\mathbf{0}\})\setminus\Gamma_{\xi_0}\big)\quad\Rightarrow\quad\langle\xi-n\rangle\geqslant C\langle n\rangle.
\end{equation}
We have, by \eqref{ineq.} (with $y=\xi$, $x=-n$, $r=2s$), 
\begin{align*}\int_{\Gamma_1}\langle\xi\rangle^{2s}|\widehat{\varphi f}(\xi)|^2\D\xi
&=\int_{\Gamma_1}\langle\xi\rangle^{2s}|(\widehat{\varphi}*\widehat{f})(\xi)|^2\D\xi\\
&=\int_{\Gamma_1}\langle\xi\rangle^{2s}\Big|\widehat{\varphi}(\xi)*\sum_{n\in\Z^d}a_n\delta(\xi+n)\Big|^2\D\xi\\
&=\int_{\Gamma_1}\langle\xi\rangle^{2s}\Big|\sum_{n\in\Z^d}a_n\widehat{\varphi}(\xi+n)\Big|^2\D\xi\\
&\leqslant\int_{\Gamma_1}\langle\xi\rangle^{2s}\Bigg(\sum_{n\in\Z^d}|a_n|^2|\widehat{\varphi}(\xi+n)|\Bigg)\Bigg(\sum_{n\in\Z^d}|\widehat{\varphi}(\xi+n)|\Bigg)\D\xi\\
&\leqslant C\int_{\Gamma_1}\Bigg(\sum_{n\in\Z^d}|a_n|^2\langle n\rangle^{2s}|\widehat{\varphi}(\xi+n)|\langle\xi+n\rangle^{2|s|}\Bigg)\D\xi\\
&=C\cdot I,
\end{align*}
where we have used that $\sum_{n\in\Z^d}|\widehat{\varphi}(\xi+n)|<+\infty$, $\xi\in\R^d$, because $\widehat{\varphi}\in\mathcal{S}(\R^d)$.
Further on,
\begin{align*}I&\leqslant\int_{\Gamma_1}\sum_{n\in\Z^d\cap\Gamma_{\xi_0}}|a_n|^2\langle n\rangle^{2s}|\widehat{\varphi}(\xi+n)|\langle\xi+n\rangle^{2|s|}\D\xi\\&\qquad+\int_{\Gamma_1}\sum_{n\in\Z^d\setminus\Gamma_{\xi_0}}|a_n|^2\langle n\rangle^{2s}|\widehat{\varphi}(\xi+n)|\langle\xi+n\rangle^{2|s|}\D\xi\\&=I_1+I_2.\end{align*}
For $I_1$, we have
\begin{align*}I_1&=\sum_{n\in\Z^d\cap\Gamma_{\xi_0}}|a_n|^2\langle n\rangle^{2s}\int_{\Gamma_1}|\widehat{\varphi}(\xi+n)|\langle\xi+n\rangle^{2|s|}\D\xi
\\&\leqslant C\sum_{n\in\Z^d\cap\Gamma_{\xi_0}}|a_n|^2\langle n\rangle^{2s}<+\infty,\end{align*}
since $\int_{\R^d}|\widehat{\varphi}(\xi+n)|\langle\xi+n\rangle^{2|s|}\D\xi\leqslant C$, $n\in\Z^d$, also because $\widehat{\varphi}\in\mathcal{S}(\R^d)$.
Concerning $I_2$, we note that $f\in\mathscr P^\prime$ implies $\sum_{n\in\Z^d}|a_n|^2\langle n\rangle^{-2m}<+\infty$ for some $m\in\mathbb{N}$. By \eqref{1*}, we have
$$\frac{\langle n\rangle^{2(m+s)}}{\langle\xi+n\rangle^{2(m+s)}}\leqslant C,\quad\xi\in\Gamma_1, \, n\in\Z^d\setminus\Gamma_{\xi_0}.$$
So
\begin{align*}I_2&=\int_{\Gamma_1}\sum_{n\in\Z^d\setminus\Gamma_{\xi_0}}|a_n|^2\langle n\rangle^{-2m}\frac{\langle n\rangle^{2(m+s)}}{\langle\xi+n\rangle^{2(m+s)}}\langle\xi+n\rangle^{2(m+s+|s|)}|\widehat{\varphi}(\xi+n)|\D\xi\\
&\leqslant C\sum_{n\in\Z^d\setminus\Gamma_{\xi_0}}|a_n|^2\langle n\rangle^{-2m}\int_{\Gamma_1}\langle\xi+n\rangle^{2(m+s+|s|)}|\widehat{\varphi}(\xi+n)|\D\xi<+\infty,
\end{align*}
where we used that $\int_{\R^d}\langle\xi+n\rangle^{2(m+s+|s|)}|\widehat{\varphi}(\xi+n)|\D\xi<+\infty$. This completes the proof.
\end{proof}

 We have the following corollary.
 \begin{corollary} Let $\varphi\in\mathcal{D}(\R^d)$ and $h\in V_s(\varphi)$, so that
 $$h(\cdot)=\sum_{k\in\Z^d}a_k\varphi(\cdot+k)\quad\text{and}\quad \sum_{k\in\Z^d\cap\Gamma}|a_k|^2\langle n\rangle^{2s}<+\infty$$
 for an open cone $\Gamma\subset\R^d\setminus\{\mathbf{0}\}$. If $\sum_{k\in\Z^d}|a_k|^2\langle k\rangle^{-2s_0}<+\infty$ for some $s_0\in\mathbb{N}$, then for every $x\in\Omega$ $(\Omega\subset\R^d$ is open set$)$ and $\xi\in\Gamma$, $(x,\xi)\notin WF_s(\widehat{h})$.
 \end{corollary}
 \begin{proof} Let $h_0=\sum_{k\in\Z^d}a_ke_k$. Then, $\widehat{h}=\widehat{\varphi}h_0$. We know by Theorem \ref{t6*} that for any $x\in\R^d$ and $\xi\in\Gamma$, $(x,\xi)\notin WF_s(h_0)$. Since the multiplication by a function in $\mathcal{S}(\R^d)$ does not decrease the set of Sobolev microlocal regular points, we conclude that $(x,\xi)\notin WF_s(\widehat{h})$.
 \end{proof}

Next, we consider the case when Lambdas are the sets of intersections of cone and $\Z^d$, where the cones are such that the projection of the wave front on the second variable is contained in them. This is an interesting case. Thus, we will take $$\Lambda^1=\Gamma_1\cap\Z^d\quad\text{and}\quad\Lambda^2=\Gamma_2\cap\Z^d,$$
so that $pr_2\big(WF_{s_1}(f_1)\big)\subset\Gamma_1$, $pr_2\big(WF_{s_2}(f_2)\big)\subset\Gamma_2$, where $pr_2(x,\xi)=\xi$, $x$, $\xi\in\R^d$. 

\begin{theorem}\label{t9} Let $f_1, f_2\in\mathscr{P}^\prime$ $($i.e. $f_1\in\mathscr P^{\prime\tau_1}$, $f_2\in\mathscr P^{\prime\tau_2})$, $\Gamma_1$ and $\Gamma_2$ be  cones of $\R^d$ so that $\Gamma_1\cap(-\Gamma_2)=\emptyset$ and that the following conditions be fulfilled.
\begin{itemize}
\item[a)] There exist $C>0$ and $\gamma\geqslant1$ such that
$$\card\{k\in\mathbb{Z}^d\,:\, n-k\in\Gamma_1 \, \wedge\, k\in\Gamma_2\}\leqslant C|n|^{\gamma}, \quad n\in\Z ^d.$$
\item[b)] Let $(x_0,\xi_0)\in\mathbb{R}^d\times(\mathbb{R}^d\setminus\{\mathbf{0}\})$ and let $\psi\in\mathcal{D}(\mathbb{T}_{\eta,x_0})$ with $\eta\in(0,1)$ and $\psi\equiv1$ in $\mathbb{T}_{\varepsilon,x_0}$, $\varepsilon<\eta$, so that
$$pr_2\big(WF_{s_1}(f_1\psi)\big)\subset\Gamma_1,\quad pr_2\big(WF_{s_2}(f_2\psi)\big)\subset\Gamma_2,$$
where $s_1\geqslant\tau_2$ and $s_2\geqslant\tau_1$.
\end{itemize}
Then, $f=(f_1\psi)_{per}(f_2\psi)_{per}$ exists in $\mathcal{D}^\prime(\R^d)$. Moreover, $f\in\mathscr{P}^\prime$.
\end{theorem}
\begin{proof} Let
$$({f_1\psi})_{per}=\sum_{k\in\mathbb Z^d}a_{1,k}e_k,  \quad ({f_2\psi})_{per}=\sum_{k\in\mathbb Z^d}a_{2,k}e_k.$$
Note, if $x\in\supp\psi$ and $\xi\in(\R^d\setminus{\{\bf{0}\}})\setminus\Gamma_1$, then $(x,\xi)\notin WF_{s_1}(f_1\psi)$. The same holds for $f_2\psi$.
Since 
$f_1\in\mathscr P^{\prime\tau_1}$ and $f_2\in\mathscr P^{\prime\tau_2}$, we know  that
\begin{equation}\label{es2}\sum_{k\in\mathbb{Z}^d\cap\Gamma_1}|a_{1,k}|^2\langle k\rangle^{-2\tau_1}<+\infty,\quad\sum_{k\in\mathbb{Z}^d\cap\Gamma_2}|a_{2,k}|^2\langle k\rangle^{-2\tau_2}<+\infty.
\end{equation}
Now as in the proof of Theorem \ref{t7}, we show that $f=(f_1\psi)_{per}(f_2\psi)_{per}$ exists and $f\in\mathscr{P}^\prime$.
\end{proof}

\begin{remark}\label{rgen}
Theorem 5.5 can be easily transferred to the case when one has several cones
$\Gamma^i_1$, $i=1,\ldots,l_1$ $($related to $f_1)$ and $\Gamma^j_2$, $j=1,\ldots,l_2$ $($related to $f_2)$
so that $\Gamma_1^i\cap\Z^d$ and $\Gamma_2^j\cap\Z^d$ contain index sets for $f_1$ and $f_2$, $i=1,\ldots,l_1$, $j=1,\ldots,l_2$ which are compatible index sets.
\end{remark}

\begin{remark}\label{r1} We will show below that under the assumption that $\Gamma_1\cap(-\Gamma_2)=\emptyset$ then condition in $a)$ of Theorem \ref{t9} holds in the case $d=2$ with $\gamma=2$. Our hypothesis is that condition $a)$ also holds $($with $\gamma=d)$ for $d\geqslant3$, but the structure of cones is more complex and we do not have the proof of this hypothesis for $d\geqslant3$.
\end{remark}

\begin{proof}[Proof of the assertion in the Remark $\ref{r1}$, for $d=2$]
We can assume that cones are acute because if it is not the case, we divide them into finite sets of such cones. So, assume that cones $\Gamma_1$ and $-\Gamma_2$ are acute and have empty intersection. By translation with vector $\overrightarrow{(0,0),(n_1,n_2)}$, there are several different positions of cones so that they have different surface of the domain laying between them. It can be equal to zero but the optimal case (maximal number of points with integer coordinates inside the intersection) is when they intersect in four points. Let us explain this case. We present the simplest position of cones (by rotations, this is not the restriction)
$$
\Gamma_1=\{(t,s): k_1t\geqslant s, t\geqslant 0 \},\quad -\Gamma_2=\{(t,s): k_2t\geqslant s, t\leqslant 0\},$$ where $k_2>k_1>0$.
Now translating $-\Gamma_2$ so that  the tip of the cone is $(n_1,n_2)$, one can calculate the points of intersection  of cones $\Gamma_1\cap\mathbb{Z}^2$ and $\big((n_1,n_2)-\Gamma_2\big)\cap\Z^2$. Coordinates of the sets of intersections, set points $A_1, A_2, A_3, A_4$ are linear combinations of the form
$$(\alpha^i_{1,1}n_1+\alpha^i_{1,2}n_2,\alpha^i_{2,1}n_1+\alpha^i_{2,2}n_2),\quad i=1,2,3,4,$$
where $\alpha^i_{j,l}$ depend on $k_1$ and $k_2$. Now, by calculating the surface of the area of the intersection of these cones, we conclude that the domain surface between two cones can be estimated by $C(n_1^2+n_2^2)$, for some $C>0$.
\end{proof}

Using Theorems \ref{t8} and \ref{t9}, we obtain the following statement.
\begin{corollary} Let $\varphi_i\in H^{s_i}$, $i=1,2$, and let $\Gamma_1$ and $\Gamma_2$ be cones so that $\Gamma_1\cap(-\Gamma_2)=\emptyset$. Assume that assertion $a)$ in Theorem $\ref{t9}$ holds.
\begin{itemize}
\item[a)] Let $x_0\in\mathbb R^d$, $f_1$, $f_2\in\mathscr D'(\mathbb R^d)$ and $\psi\in\mathscr D(\mathbb{T}_{\eta,x_0})$ with $\eta\in(0,1)$ so that
$\psi\equiv1$ in $\mathbb{T}_{\varepsilon,x_0}$, $\varepsilon<\eta.$
Assume that
$$\widehat{({f_1\psi})}=\widehat{\varphi_1}\sum_{k\in\mathbb Z^d}a_ke_k,  \quad \widehat{({f_2\psi})}=\widehat{\varphi_2}\sum_{k\in\mathbb Z^d}b_ke_k,$$
and that \eqref{es2} holds for $\sum_{k\in\mathbb Z^d}a_ke_k$ and $\sum_{k\in\mathbb Z^d}b_ke_k$.
Moreover, assume that condition $b)$ of Theorem $\ref{t9}$ holds.
Then, there exists $s\in\R$ so that
$$f_1\psi(\cdot)=\sum_{k\in\mathbb Z^d}a_k\varphi_1(\cdot+k),\quad  f_2\psi(\cdot)=\sum_{k\in\mathbb Z^d}b_k\varphi_2(\cdot+k)$$
are elements of $V_s(\varphi_1)$, $V_s(\varphi_2)$, respectively, and their product $(f_1\psi)*(f_2\psi)\in V_s(\varphi_1*\varphi_2)$.
\item[b)] Let $g_i\in V_{s_i}(\varphi_i)$, $i=1,2$, and let $(x_0,\xi_0)\in\mathbb{R}^d\times(\mathbb{R}^d\setminus\{\mathbf{0}\})$.
For $\widehat{g}_i=\widehat{\varphi}_if_i$ we assume that $pr_2\big(WF_{s_i}(f_i)\big)\subset\Gamma_i$, $i=1,2$. Moreover, we assume that \eqref{es2} holds, where $s_1\geqslant\tau_2\geqslant0$ and $s_2\geqslant\tau_1\geqslant0$. Then, there exists $s\in\R$ so that
$$g=g_1*g_2\in V_{s}(\varphi_1*\varphi_2)\quad\mbox{and} \quad g(\cdot)=\sum_{n\in\mathbb{Z}^d}a_n(\varphi_1*\varphi_2)(\cdot+n).$$
\end{itemize}
\end{corollary}

\vskip1cm
{\bf Acknowledgment.}
S. Pilipovi\' c is supported by the Serbian Academy of Sciences and
Arts, project F10. A. Aksentijevi\' c and S. Aleksi\' c were supported by the
Science Fund of the Republic of Serbia, $\#$GRANT No $2727$, {\it Global
and local analysis of operators and distributions} - GOALS and they are gratefully acknowledge the financial support of the Ministry of
Science, Technological Development and Innovation of the Republic
of Serbia (Grants No. 451-03-66/2024-03/200122 \& 451-03-65/2024-03/200122).



\begin{thebibliography}{33}



\bibitem{AF} R. A. Adams and J. J. F. Fournier, {\it Sobolev Spaces}, (Second edition), Academic Press, Elsevier, 2003.

\bibitem{A1} A. Aguilera, C. Cabrelli, D. Carbajal and V. Paternostro, {\it Diagonalization of shift-preserving operators}, Advan. in Math. {\bf389} (2021), paper No. 107892, 32 pp.

\bibitem{A2} A. Aguilera, C. Cabrelli, D. Carbajal and V. Paternostro, {\it Dynamical sampling for shift-preserving operators},
 Appl. Comput. Harmon. Anal. {\bf51} (2021), 258--274.

\bibitem{aap} A. Aksentijevi\' c, S. Aleksi\' c, and S. Pilipovi\' c, {\it The structure of shift-invariant subspaces of Sobolev
spaces}, {Theor Math Phys} {\bf218} (2024), 177--191.

\bibitem{AG} A. Aldroubi and K. Gr\"ochenig, {\it Non-uniform sampling and reconstruction in shift-invariant spaces}, SIAM Rev. {\bf43} (2001), 585--620.

\bibitem{AST} A. Aldroubi, Q. Sun and W. Tang, {\it $p$-frames and shift-invariant subspaces of $L^p$}, J. Fourier Anal. Appl. {\bf7} (2001), 1--21.

\bibitem{AMS} P. Antosik, J. Mikusi\' nski and R. Sikorski, {\it Theory of Distributions. The Sequential Approach},
 Elsevier Scientific Publishing Co., Amsterdam; PWN---Polish Scientific Publishers, Warsaw, 1973.

\bibitem{Beals} R. Beals, {\it Advanced Mathematical Analysis. Periodic Functions and Distributions, Complex Analysis, Laplace Transform and Applications}, Graduate Texts in Mathematics, No. 12, Springer-Verlag, New York-Heidelberg, 1973.

\bibitem{BVR1} C. de Boor, R. A. DeVore and A. Ron, {\it The structure of finitely generated shift-invariant spaces in $L^2(\mathbb R^d)$}, J. Funct. Anal. {\bf119} (1994), 37--78.

\bibitem{BVR2} C. de Boor, R. A. DeVore and A. Ron, {\it Approximation from shift-invariant subspaces of $L^2(\mathbb R^d)$}, Trans. Amer. Math. Soc. {\bf341} (1994), 787--806.

\bibitem{MB} M. Bownik, {\it The structure of shift-invariant subspaces of $L^2(\mathbb{R}^n)$}, J. Funct. Anal. {\bf177} (2000), 282--309.

\bibitem{MRken} M. Bownik and K. A. Ross, {\it The structure of translation-invariant spaces on locally compact abelian groups}, J. Fourier Anal. Appl. {\bf21} (2016), 849--884.

\bibitem{BR} M. Bownik and Z. Rzeszotnik, {\it The spectral function of shift-invariant spaces}, Michigan Math. J. {\bf51} (2003), 387--414.

\bibitem{DV} A. Debrouwere and J. Vindas, {\it Discrete characterizations of wave front sets of Fourier-Lebesgue and quasianalytic type}, J. Math. Anl. Appl. {\bf438} (2016), 889--908.

\bibitem{KG} K. Gr\"ochening, {\it Foundations of Time-Frequency Analysis}, Birkhauser, Boston, 2001.

\bibitem{H} H. Helson, {\it Lectures on Invariant Subspaces}, Academic Press, New York, London, 1964.

\bibitem{Hor} L. H\"{o}rmander, {\it The Analysis of Linear Partial Differential Operators I: Distribution Theory and Fourier Analysis}, Springer-Verlag, 1983.

\bibitem{Hor2} L. H\"{o}rmander, {\it Lectures on Nonlinear Hyperbolic Differential Equations}, Springer-Verlag, 1997.

\bibitem{BLZ} B. Liu, R. Li and Q. Zhang, {\it The structure of finitely generated shift-invariant spaces in mixed Lebesgue spaces $L^{p,q}(\mathbb R^{d+1})$}, Banach J. Math. Anal. {\bf14} (2020), 63--77.

\bibitem{MPSV} S. Maksimovi\' c, S. Pilipovi\' c, P. Sokoloski and J. Vindas, {\it Wave fronts via Fourier series coefficients},
Publ. Inst. Math. (Beograd) (N.S.) {\bf97} (2015), 1--10.

\bibitem{SS} S. Pilipovi\' c and S. Simi\' c, {\it Frames for weighted shift-invariant spaces}, Mediterr. J. Math. {\bf9} (2012), 897--912.

\bibitem{RS} A. Ron and Z. Shen, {\it Frames and stable bases for shift-invariant subspaces of $L^2(\mathbb{R}^d)$}, Canad. J. Math. {\bf47} (1995), 1051--1094.

\bibitem{Ruzh} M. Ruzhansky and V. Turunen, {\it Pseudo-Differential Operators and Symmetries. Background Analysis and Advanced Topics}, Birkh\" auser Verlag, Basel, 2010.

\bibitem{LS} L. Schwartz, {\it Th\' eorie des Distributions}, Hermann, Paris, 1966.

\bibitem{SSS} C. E. Shin and Q. Sun, {\it Stability of Localized operators}, J. Funct. Anal. {\bf256} (2009), 2417--2439.

\bibitem{VSV} V. S. Vladimirov, {\it Generalized Functions in Mathematical Physics, Transl. by G. Yankovskii}, "Mir'', Moscow, xii+362 pp, 1979.
\end{thebibliography}
\end{document}